\documentclass[11pt,leqno]{article}
\usepackage[margin=1in]{geometry} 
\usepackage{amssymb,amsfonts,amsmath,bbm,mathrsfs,stmaryrd,mathtools}
\usepackage{xcolor}
\usepackage{url}

\usepackage{extarrows}

\usepackage[shortlabels]{enumitem}
\usepackage{tensor}

\usepackage{xr}
\usepackage[T1]{fontenc}
\usepackage[utf8]{inputenc}
\usepackage[colorlinks,
linkcolor=black!75!red,
citecolor=blue,
pdftitle={},
pdfproducer={pdfLaTeX},
pdfpagemode=None,
bookmarksopen=true,
bookmarksnumbered=true,
backref=page]{hyperref}

\usepackage{tikz}
\usetikzlibrary{arrows,calc,decorations.pathreplacing,decorations.markings,decorations.shapes,intersections,shapes.geometric,through,fit,shapes.symbols,positioning,decorations.pathmorphing}

\makeatletter
\newlength\zig@L
\newlength\zig@La
\newlength\zig@Lb

\newcommand{\xzigrightarrow}[2][]{%
  \mathrel{%
    \settowidth{\zig@La}{$\scriptstyle #2$}%
    \settowidth{\zig@Lb}{$\scriptstyle #1$}%
    \zig@L=\zig@La\relax
    \ifdim\zig@Lb>\zig@L \zig@L=\zig@Lb\fi
    \advance\zig@L by 2.2em\relax
    \tikz[baseline=-0.65ex]{%
      \draw[->,
      line cap=round,
      decorate,
      decoration={zigzag,segment length=4pt,amplitude=1.1pt}]%
      (0,0) -- (\zig@L,0)
      node[midway,above=2pt] {$\scriptstyle #2$}%
      \if\relax\detokenize{#1}\relax\else
        node[midway,below=2pt] {$\scriptstyle #1$}%
      \fi
      ;
    }%
  }%
}
\makeatother

\makeatletter
\newcommand{\squigjoin}{1mu} % tune this: 0.5mu, 1mu, 1.5mu, ...

\def\sqleft@{\sim}                    % no overlap here
\def\sqmid@{\sim\mkern-\squigjoin}    % overlap only between repeated mids

\def\rightsquigarrowfill@{%
  \arrowfill@{\sqleft@}{\sqmid@}{\mkern-4mu\succ}%
}

\newcommand{\xrightsquigarrow}[2][]{%
  \ext@arrow 0359\rightsquigarrowfill@{#1}{#2}%
}
\makeatother

\makeatletter
\newcommand*\circled[1]{\tikz[baseline=(char.base)]{
    \node[shape=circle, draw, inner sep=0pt, 
    minimum height={\f@size},] (char) {\vphantom{WAH1g}#1};}}
\makeatother

\makeatletter
\DeclareRobustCommand\widecheck[1]{{\mathpalette\@widecheck{#1}}}
\def\@widecheck#1#2{%
  \setbox\z@\hbox{\m@th$#1#2$}%
  \setbox\tw@\hbox{\m@th$#1%
    \widehat{%
      \vrule\@width\z@\@height\ht\z@
      \vrule\@height\z@\@width\wd\z@}$}%
  \dp\tw@-\ht\z@
  \@tempdima\ht\z@ \advance\@tempdima2\ht\tw@ \divide\@tempdima\thr@@
  \setbox\tw@\hbox{%
    \raise\@tempdima\hbox{\scalebox{1}[-1]{\lower\@tempdima\box
        \tw@}}}%
  {\ooalign{\box\tw@ \cr \box\z@}}}
\makeatother

\usepackage{braket}

\usepackage[amsmath,thmmarks,hyperref]{ntheorem}
\usepackage{cleveref}

\newcommand\nthalias[1]{\AddToHook{env/#1/begin}{\crefalias{lemma}{#1}}}

\nthalias{definition}
\nthalias{example}
\nthalias{examples}
\nthalias{remark}
\nthalias{remarks}
\nthalias{convention}
\nthalias{notation}
\nthalias{construction}
\nthalias{sketch}
\nthalias{theoremN}
\nthalias{propositionN}
\nthalias{corollaryN}
\nthalias{lemma}
\nthalias{proposition}
\nthalias{corollary}
\nthalias{theorem}
\nthalias{conjecture}
\nthalias{question}
\nthalias{assumption}

\creflabelformat{enumi}{#2#1#3}

\crefname{section}{Section}{Sections}
\crefformat{section}{#2Section~#1#3} 
\Crefformat{section}{#2Section~#1#3} 

\crefname{subsection}{\S}{\S\S}
\AtBeginDocument{%
  \crefformat{subsection}{#2\S#1#3}%
  \Crefformat{subsection}{#2\S#1#3}% 
}

\crefname{subsubsection}{\S}{\S\S}
\AtBeginDocument{%
  \crefformat{subsubsection}{#2\S#1#3}%
  \Crefformat{subsubsection}{#2\S#1#3}% 
}

\theoremstyle{plain}

\newtheorem{lemma}{Lemma}[section]
\newtheorem{proposition}[lemma]{Proposition}

\newtheorem{theorem}[lemma]{Theorem}

\theoremstyle{plain}
\theoremnumbering{Alph}

\theoremstyle{plain}
\theorembodyfont{\upshape}
\theoremsymbol{\ensuremath{\blacklozenge}}

\newtheorem{definition}[lemma]{Definition}
\newtheorem{example}[lemma]{Example}

\newtheorem{remark}[lemma]{Remark}

\crefname{definition}{definition}{definitions}
\crefformat{definition}{#2definition~#1#3} 
\Crefformat{definition}{#2Definition~#1#3} 

\crefname{ex}{example}{examples}
\crefformat{example}{#2example~#1#3} 
\Crefformat{example}{#2Example~#1#3} 

\crefname{exs}{example}{examples}
\crefformat{examples}{#2example~#1#3} 
\Crefformat{examples}{#2Example~#1#3} 

\crefname{remark}{remark}{remarks}
\crefformat{remark}{#2remark~#1#3} 
\Crefformat{remark}{#2Remark~#1#3} 

\crefname{remarks}{remark}{remarks}
\crefformat{remarks}{#2remark~#1#3} 
\Crefformat{remarks}{#2Remark~#1#3} 

\crefname{convention}{convention}{conventions}
\crefformat{convention}{#2convention~#1#3} 
\Crefformat{convention}{#2Convention~#1#3} 

\crefname{notation}{notation}{notations}
\crefformat{notation}{#2notation~#1#3} 
\Crefformat{notation}{#2Notation~#1#3} 

\crefname{table}{table}{tables}
\crefformat{table}{#2table~#1#3} 
\Crefformat{table}{#2Table~#1#3}

\crefname{lemma}{lemma}{lemmas}
\crefformat{lemma}{#2lemma~#1#3} 
\Crefformat{lemma}{#2Lemma~#1#3} 

\crefname{proposition}{proposition}{propositions}
\crefformat{proposition}{#2proposition~#1#3} 
\Crefformat{proposition}{#2Proposition~#1#3} 

\crefname{propositionN}{proposition}{propositions}
\crefformat{propositionN}{#2proposition~#1#3} 
\Crefformat{propositionN}{#2Proposition~#1#3} 

\crefname{corollary}{corollary}{corollaries}
\crefformat{corollary}{#2corollary~#1#3} 
\Crefformat{corollary}{#2Corollary~#1#3} 

\crefname{corollaryN}{corollary}{corollaries}
\crefformat{corollaryN}{#2corollary~#1#3} 
\Crefformat{corollaryN}{#2Corollary~#1#3} 

\crefname{theorem}{theorem}{theorems}
\crefformat{theorem}{#2theorem~#1#3} 
\Crefformat{theorem}{#2Theorem~#1#3} 

\crefname{theoremN}{theorem}{theorems}
\crefformat{theoremN}{#2theorem~#1#3} 
\Crefformat{theoremN}{#2Theorem~#1#3} 

\crefname{enumi}{}{}
\crefformat{enumi}{#2#1#3}
\Crefformat{enumi}{#2#1#3}

\crefname{assumption}{assumption}{Assumptions}
\crefformat{assumption}{#2assumption~#1#3} 
\Crefformat{assumption}{#2Assumption~#1#3} 

\crefname{construction}{construction}{Constructions}
\crefformat{construction}{#2construction~#1#3} 
\Crefformat{construction}{#2Construction~#1#3} 

\crefname{sketch}{sketch}{Sketches}
\crefformat{sketch}{#2sketch~#1#3} 
\Crefformat{sketch}{#2Sketch~#1#3} 

\crefname{question}{question}{Questions}
\crefformat{question}{#2question~#1#3} 
\Crefformat{question}{#2Question~#1#3} 

\crefname{equation}{}{}
\crefformat{equation}{(#2#1#3)} 
\Crefformat{equation}{(#2#1#3)}

\numberwithin{equation}{section}

\theoremstyle{nonumberplain}
\theoremsymbol{\ensuremath{\blacksquare}}

\newtheorem{proof}{Proof}
\newcommand\pf[1]{\newtheorem{#1}{Proof of \Cref{#1}}}

\newcommand\bB{{\mathbb B}}
\newcommand\bC{{\mathbb C}}

\newcommand\bG{{\mathbb G}}

\newcommand\bP{{\mathbb P}}

\newcommand\bR{{\mathbb R}}
\newcommand\bS{{\mathbb S}}

\newcommand\bZ{{\mathbb Z}}

\newcommand\cP{{\mathcal P}}

\DeclareMathOperator{\spn}{\mathrm{spn}}

\DeclareMathOperator{\cvx}{\mathrm{cvx}}

\newcommand{\qedhere}{\mbox{}\hfill\ensuremath{\blacksquare}}

\title{Grassmannian spines, projection closure operators, and diametric sweeps}
\author{Alexandru Chirvasitu}

\begin{document}

\date{}

\newcommand{\Addresses}{{% additional braces for segregating \footnotesize
    \bigskip
    \footnotesize

    \textsc{Department of Mathematics, University at Buffalo}
    \par\nopagebreak
    \textsc{Buffalo, NY 14260-2900, USA}  
    \par\nopagebreak
    \textit{E-mail address}: \texttt{achirvas@buffalo.edu}

    % % \medskip
    % % 
    % % \textsc{Department of Mathematics, INSTITUTION}
    % % \par\nopagebreak
    % % \textsc{ADDRESS}
    % % \par\nopagebreak
    % % \textit{E-mail address}: \texttt{??}
    % % 

  }}

\maketitle

\begin{abstract}
  For positive integers $r<d<n$ equip the powerset $2^{\mathbb{G}(r,V)}$ of the $r$-plane Grassmannian of an $n$-dimensional Hilbert space with the closure operator attaching to a set of $r$-planes the smallest superset which along with two $r$-planes also contains all $r$-dimensional orthogonal projections of one onto any $d$-plane containing the other. In the regime $2r\le d$ the classification of closed subsets of $\mathbb{G}(r,V)$ rigidifies, these being precisely the sets of $r$-planes containing a fixed $(\le r)$-plane. The result generalizes its $(r,d,n)=(1,2,3)$ instance, of use in recent geometric-rigidity results motivated by matrix preserver problems. 

  An auxiliary result classifies the balls centered at $p_0\in \mathbb{R^d}$ as the compact fixed points of the dynamical system transforming $K\subseteq \mathbb{R}^d$ into its $p_0$-based diametric sweep: the union of all diameter-$p_0p$ balls for $p\in K$. 
\end{abstract}


\vspace{.5cm}

\noindent{MSC 2020: 51A05; 46C05; 51A20; 15A04; 51M15; 15A86
  
  % 51A05 General theory of linear incidence geometry and projective geometries
  % 46C05 Hilbert and pre-Hilbert spaces: geometry and topology (including spaces with semidefinite inner product)
  % 51A20 Configuration theorems in linear incidence geometry
  % 15A04 Linear transformations, semilinear transformations
  % 51M15 Geometric constructions in real or complex geometry
  % 15A86 Linear preserver problems
}

% \tableofcontents

%%%%%%%%%%%%%%%%%%%%%%%%%%%%%%%% 
%%%%%%%%%%%%%%%%%%%%%%%%%%%%%%%% 
\section*{Introduction}

The results forming the focus of the present note originated as offshoots of considerations touching on several entwined themes:
\begin{itemize}[wide]
\item \emph{preserver problems}, in the spirit of, say, \cite{2501.06840v2, MR4927632,GogicPetekTomasevic,Petek-HM,zbMATH05302134} (for only a minuscule glimpse of a vast literature), to the effect that maps between matrix spaces preserving various invariants, prominently spectra and commutation, tend to fall into several rigidly-constrained classes: conjugations, transpose conjugations, and, depending on the specifics of the problem, a handful of other classifiable ``exotic'' types;
  
\item \emph{Wigner}-type results \cite[Chapter 4]{pank_wign}, spiraling out of foundational work on quantum mechanics (see e.g. \cite[\S 3.1]{ball_qm_2e_2015}): all in a common spirit of classifying self-maps of projectivized Hilbert spaces $\bP H$ (i.e. the spaces of \emph{states} familiar in quantum physics) preserving various quantum-mechanically meaningful invariants (orthogonality, say, interpretable as observer \emph{commeasurability} \cite[\S 2-2]{mack_qm_1963}, or angles, cast as state-transition probabilities \cite[(2.8)]{ball_qm_2e_2015}) are again tightly classifiable (induced by unitary or conjugate-unitary operators, say);
  
\item and projective incidence geometry, in a sense emerging as foundational to both themes just touched on, and epitomized by the \emph{fundamental theorem of projective geometry} \cite[Theorem 3.1]{zbMATH01747827}: structure-preserving maps between projective spaces are induced by semilinear transformations. 
\end{itemize}

\cite[Theorem 0.1]{2601.11455v2} will serve as a concrete entry point to the subsequent material: maps $\bG V\xrightarrow{\Psi}\bG W$ between Grassmannians of $(\ge 3)$-dimensional Hilbert spaces, respecting dimensions and lattice operations for \emph{commeasurable} pairs of spaces, are induced by semilinear injections. 

Commeasurability here means that the two spaces in question are both spans of elements ranging over a common orthonormal basis, i.e. may be interpreted as sharing a quantum-mechanical \emph{context} \cite[\S B.1]{MR4449330} (hence the term, intended as reminiscent of observable measurement). In the course of the result's proof, it becomes apparent that
\begin{itemize}[wide]
\item the crux of the matter is the injectivity on lines $\ell\in \bP V$, $\dim V=\dim W=3$ of the projective-spaces maps in question;
\item and, were injectivity to fail via $\Psi \ell=\Psi\ell'$, $\ell\ne \ell'$, the aforementioned contextuality would force the same $\Psi$-value on all $\ell''$ with $\ell\oplus \ell''$ and $\ell'\oplus \ell''$ commeasurable. 
\end{itemize}
It is at this point that the following construct becomes relevant.

Working over fields $\Bbbk$ frequently specialized to $\bR$ or $\bC$, write $\bG(r,V)$ for the \emph{Grassmannian} \cite[\S 1.2]{pank_wign} of $r$-dimensional subspaces of a vector space $V$, and $\bP V:=\bG(1,V)$ for the projective space of lines therein. Assuming $V$ a finite-dimensional real or complex Hilbert space, define ternary relations on $\bG(r,V)$ by
\begin{equation*}
  \tau_d\left(\zeta,\zeta'\mid\zeta''\right)
  \xLeftrightarrow{\quad}
  \exists\left(\text{$d$-plane $\pi\ge \zeta$}\right)
  \left(\zeta''=P_{\pi}\zeta'\right),
  \quad
  1\le d\le \dim V,
\end{equation*}
$P_{\bullet}$ denoting the orthogonal projection onto the space $\bullet\le V$. We also need:

\begin{definition}\label{def:sat}
  A subset $S\subseteq \bG(r,V)$ is \emph{$\tau_d$-saturated}, or \emph{saturated with respect to $\tau_d$} if
  \begin{equation*}
    \forall\left(\zeta,\zeta'\in S\subseteq \bG(r,V)\right)
    \forall\left(\zeta''\in \bG(r,V)\right)
    \left(
      \tau_d\left(\zeta,\zeta'\mid \zeta''\right)
      \xRightarrow{\quad}
      \zeta''\in S
    \right).
  \end{equation*}
  Or: $S$ contains elements in the codomain of $\tau_d$ whenever it contains the relevant elements in its domain. 
\end{definition}

The preceding discussion can be summarized as the observation that as soon as $\Psi\ell=\Psi\ell'$, $\Psi$ in fact takes a common value on the entire $\tau_2$-saturation of $\{\ell,\ell'\}$; hence the natural question of just how large saturated sets can be. The following rigidity result confirms that they are in rather short supply.

\begin{theorem}\label{th:sat.1.all}  
  Let $1<d< n\in \bZ_{\ge 3}$ and $V$ an $n$-dimensional real or complex Hilbert space.

  The only $\tau_d$-saturated subsets of $\bP V$ are singletons and $\bP V$ itself. 
\end{theorem}

As complete a classification as it may be within its domain, the discrete dichotomy in \Cref{th:sat.1.all} (saturation entails being either singleton or full) obscures the deeper nature of the rigidity phenomenon underlying that classification. \Cref{ex:spine} notes that more generally, the full set of $r$-planes containing a fixed common $(\le r)$-dimensional core (the paper title's \emph{Grassmannian spines}) is $\tau_d$ saturated. The converse is the higher analogue of \Cref{th:sat.1.all}: per \Cref{th:2rd.sat.r.all} below, the spines are precisely the $\tau_d$-saturated sets under only the plainly necessary dimensional requirements. 

In roughly the same spirit of geometric/metric rigidity in the context of closure operators, there is an auxiliary result to \Cref{th:sat.1.all} that although simple in nature, is suggestive of ramifications. Consider, for fixed $p_0\in \bR^d$, the set-valued discrete dynamical system attaching to a subset $K\subseteq \bR^d$ its \emph{$p_0$-based diametric sweep}: the union of the balls with diameters $p_0 p$, $p\in K$. \Cref{pr:diam.cl}, distilling some of the elementary-geometric content pertinent to $\tau$-saturation, classifies the compact fixed points of that dynamical system as precisely the $p_0$-centered balls.  

%%%%%%%%%%%%%%%%%%%%%%%%%%%%%%%% 
\subsection*{Acknowledgments}

I am grateful for instructive comments from I. B\'ar\'any and R. Schneider.

% % %%%%%%%%%%%%%%%%%%%%%%%%%%%%%%%% 
% % %%%%%%%%%%%%%%%%%%%%%%%%%%%%%%%% 
% % \section{Preliminaries}\label{se:prel}
% % 

%%%%%%%%%%%%%%%%%%%%%%%%%%%%%%%% 
%%%%%%%%%%%%%%%%%%%%%%%%%%%%%%%% 
\section{Orthogonality-oriented relations on state spaces}\label{se:orth.rel}

The following low-dimension result is auxiliary to proving \Cref{th:sat.1.all}, which will subsequently supersede it.

\begin{proposition}\label{pr:sat.1.23}
  For a $3$-dimensional real or complex Hilbert space $V$ the only $\tau_2$-saturated subsets of $\bP V$ are singletons and $\bP V$ itself. 
\end{proposition}

A brief detour will help build some auxiliary tooling. I have been unable to locate the following result in precisely its present form; it is recorded here for completeness, future use, and whatever intrinsic interest it may possess.

\begin{proposition}\label{pr:diam.cl}
  Let $p_0\in \bR^d$. A compact subset $K\subset \bR^d$ containing all balls $\bB_{p_0,p}$ with diameters $p_0 p$ for arbitrary $p\in K$ is a ball centered at $p_0$. 
\end{proposition}

\begin{remark}\label{re:it.diam}
  \Cref{pr:diam.cl} can be cast in dynamical terms, as a set-valued fixed-point result. To elaborate, define the \emph{($p_0$-based) diameter} transform of a set $K\subseteq \bR^d$ by
  \begin{equation*}
    \Delta_{p_0}
    K:=
    \bigcup_{p\in K}\bB_{p_0,p}
    ,\quad
    \bB_{p_0,p}:=\text{ball with diameter $p_0p$}.
  \end{equation*}
  Such geometrically-motivated transforms, often amenable to iteration, abound in the mathematical morphology literature (e.g. dilation/erosion/component-extraction transforms of \cite[\S II.B]{serra_im-an-1_1984}).

  In the present context, \Cref{pr:diam.cl} classifies the compact fixed points of $\Delta_{p_0}$ as the $p_0$-centered balls. 
\end{remark}

\begin{example}\label{ex:cardio}
  Per \cite[Lemma 2 and subsequent paragraph]{MR3324686}, the diameter transform $\Delta_{p_0}\bS^1$ of a circle $\bS^1\ni p_0$ is (the region enclosed by) the \emph{cardioid} with a cusp at $p_0$ and touching $\bS^1$ at the antipodal point. 
\end{example}

\pf{pr:diam.cl}
\begin{pr:diam.cl}
  It suffices to argue the case $d=2$: if all intersections $K\cap\pi$ for 2-planes $\pi\ni p_0$ are $p_0$-centered balls so too is $K$. 
  
  \begin{enumerate}[(I),wide]
  \item\label{item:pr:diam.cl:pf.is.cvx}\textbf{: The convex hull $\cvx K$ is strictly convex.} Recall \cite[Chapter One, Definition post Theorem 1]{zbMATH03481242} that \emph{strict} convexity, for a convex set such as $\cvx K$, means precisely that the boundary $\partial \cvx K$ contains no non-degenerate segments.
    
    Were this otherwise, there would be \emph{extreme} points $p^+\ne p^-\in K\cap \partial\cvx K$ (in the usual convexity-parlance sense \cite[post Lemma 1.4.2]{schn_cvx_2e_2014}) connected by a segment $p^+p^-\subset \partial \cvx K$. At least one of the balls (i.e. disks) $B_{p_0,p^{\pm}}$ will intersect both half-planes cut out by the $\cvx K$-\emph{supporting} line \cite[\S 1.3]{schn_cvx_2e_2014} $p^+p^-$
    \begin{itemize}[wide]
    \item certainly if $p_0\in \{p^{\pm}\}$, for then the single line in question contains the diameter of the single non-degenerate disk;
      
    \item and also when $p_0\not\in \{p^{\pm}\}$, for in that case the only $\bB_{p_0,p^{\pm}}$-supporting lines respectively at $p^{\pm}$ are orthogonal to the distinct lines $p_0 p^{\pm}$, so cannot both coincide with $p^+ p^-$.
    \end{itemize}
    This contradicts the disk-containment assumption, settling the present step \Cref{item:pr:diam.cl:pf.is.cvx}. Observe also, in consequence, that in fact $K=\cvx K$ itself is (strictly) convex.

  \item\textbf{: $\partial K$ is differentiable away from $p_0$.} In other words, each $q\in \partial K$ is incident to a unique $K$-supporting line unless $q=p_0$ (which possibility we have not discounted yet, but will fall by the wayside in the process of the argument). Indeed: the disk $\bB_{p_0,q}$, being smooth itself, has a unique supporting line at $q$; consequently, of any two $K$-supporting lines intersecting at $q\in \partial K$ one must cross the interior $\overset{\circ}{\bB}_{p_0,q}$.

  \item\textbf{: Conclusion.} Parametrize the curve $\partial K\setminus \{p_0\}$ as $\mathbf{r}(t)$ in a $p_0$-centered coordinate system. Every disk $\bB_{p_0,q:=\mathbf{r}(t)}$ must be supported by the $K$-supporting line through $q$, so that $\dot{\mathbf{r}}(t)\cdot \mathbf{r}(t)=0$. This renders $|\mathbf{r}(t)|$ constant, finishing the proof.  \qedhere
  \end{enumerate}
\end{pr:diam.cl}

\begin{remark}\label{re:reg.if.reg}
  The simple observation underpinning the differentiability noted in the course of the proof of \Cref{pr:diam.cl} is that a point $p\in \partial K$ for compact convex $K$ will be \emph{regular} (i.e. \cite[\S 2.2]{schn_cvx_2e_2014} incident to a unique $K$-supporting hyperplane) provided some regular compact convex $K_p$ is contained in $K$: naturally, all $K$-supporting hyperplanes must also be $K_p$-supporting at $p$.
\end{remark}

\pf{pr:sat.1.23}
\begin{pr:sat.1.23}
  In other words, saturated sets containing at least two distinct lines must be full. Observe that nothing is lost if, for convenience, we work over $\bR$: having fixed two $\ell\ne \ell'\in \bP V$ belonging to the saturated set $S$, one can always select a \emph{real structure} \cite[Definition 18.5.20]{jones_vna} on $V$ compatible with both $\ell$ and $\ell'$.

  Consider, then, lines $\ell,\ell'$ in a $\tau_2$-saturated set $S\subseteq \bP V$ for real 3-dimensional $V$, both distinct and non-orthogonal (for indeed, $\ell'$ will have some projection $P_{\pi}\ell'$, $\ell \le\pi\in \bG(2,V)$ non-orthogonal to $\ell$). For fixed $0\ne p'\in \ell'$ the orthogonal-projection locus
  \begin{equation*}
    \left\{P_{\pi}p'\ :\ \ell\le \pi\in \bG(2,V)\right\} 
  \end{equation*}
  is the circle $C_0$ with diameter $pp'$, $p:=P_{\ell}p'\ne 0$, lying in the affine plane $\ell^{\perp}+p$: containing $p$ (and $p'$) and orthogonal to $\ell$. The procedure just carried out functions with $\ell$ and any of the lines
  \begin{equation*}
    \ell'':=0p_0
    ,\quad
    p_0\in C_0\ne p
  \end{equation*}
  in place of $\ell'$, so that for all members $q_1$ of the region $D_1\subset \ell^{\perp}+p$ swept out by the circles (equivalently, the disks) with diameters $pp_0$, $p_0\in C_0$ the lines $0q_1$ all belong to $S$. In the language of \Cref{re:it.diam}, $D_1=\Delta_p C_0$ (the $\Delta$ constructions being carried out in the affine plane $\ell^{\perp}+p$). We can now iterate indefinitely:
  \begin{equation*}
    \forall\left(q\in D_{\infty}\right)
    \left(0q\in S\right)
    ,\quad
    D_{\infty}:=\bigcup_{n\ge 1}D_n
    ,\quad
    D_{n+1}:=\Delta_p D_n.
  \end{equation*}
  It follows from \Cref{pr:diam.cl} above that $\overline{D_{\infty}}\subset \ell^{\perp}+p$ is the radius-$pp'$, $p$-centered disk in that affine plane, and contains the interior of that disk. Thus: for lines $\ell\ne \ell'\in S$ forming an angle $0<\theta<\frac{\pi}{2}$, the lines interior to the circular cone (with tip angle $2\theta$) formed by revolving $\ell'$ around $\ell$ again belong to $S$. This will suffice to conclude:
  \begin{itemize}[wide]
  \item $S\subseteq \bP V$ is on the one hand open, as we have just argued that along with $\ell\ne \ell'$ it contains an entire revolution cone around $\ell$; 
  \item on the other hand, $S$ is also closed: if
    \begin{equation*}
      S
      \ni
      \ell_n
      \xrightarrow[\quad n\quad]{\quad}
      \ell
      \in \bP V
    \end{equation*}
    then some $\ell_m\ne \ell$ will form a $\left(<\frac{\pi}2\right)$-angle with $\ell$, and revolving $\ell_m$ around $\ell_n$ sufficiently close to $\ell$ will produce a cone whose interior contains $\ell$.
  \end{itemize}
  $\bP V$ being connected, a non-empty clopen subset thereof must be full. 
\end{pr:sat.1.23}

\begin{remark}\label{re:tau.sym}
  In the context of the small-dimension case treated in \Cref{pr:sat.1.23}, the relation $\tau_2$ in $\bP V$ is symmetric in its first two variables: for $\ell\ne \ell'$
  \begin{equation*}
    \tau_2(\ell,\ell'\mid\ell'')
    \iff
    \ell''
    =
    \left(\pi:=\spn\left\{\ell,\ell''\right\}=\ell\oplus \pi^{\perp}\right)
    \cap
    \left(\pi':=\spn\left\{\ell',\ell''\right\}=\ell'\oplus {\pi'}^{\perp}\right),
  \end{equation*}
  with the caveat that $\spn\left\{\ell,\ell''\right\}$ is to be interpreted as $\ell\oplus \left(\ell\oplus \ell'\right)^{\perp}$ in the degenerate case $\ell=\ell''$ and similarly with the pair $\ell',\ell''$ in place of $\ell,\ell''$ (cf. \cite[Proof of Theorem 0.1]{2601.11455v2}).
 
  Even this partial symmetry is not valid in general, for $\tau_d$ on $\bG(r,V)$: one might have $\tau_d(\eta,\eta'\mid \eta)$ if $P_{\pi}\eta'=\eta$, while at the same time $\dim \left(\eta+\eta'\right)>d$ so that no $d$-plane can contain $\eta$ and $\eta'$ both, meaning that $\tau_d(\eta',\eta\mid \eta)$ fails. Examples are easily constructed with $r=2$ and $d=3$, say. 
\end{remark}

\pf{th:sat.1.all}
\begin{th:sat.1.all}
  \Cref{pr:sat.1.23} does most of the work, the remainder being an inductive argument that will remove the dimension-related distractions. Denote the statement's claim by $\cP_{d,n}$ (always assuming $1<d<n$, whenever employing the symbol). 
  \begin{enumerate}[(I),wide]
  \item\label{item:th:sat.all:pf.dnd1n1}\textbf{: $\cP_{d,n}\Rightarrow \cP_{d+1,n+1}$.} For a saturated set
    \begin{equation*}
      \bP V
      \supseteq
      S\ni \ell\ne \ell'
      ,\quad
      \ell,\ell'\in \bP V
      ,\quad
      \dim V=n+1,
    \end{equation*}
    substitute
    \begin{itemize}[wide]
    \item  for $V$, a hyperplane $H\le V$ containing $\ell$, $\ell'$ and any arbitrary third line $\ell''$;

    \item and for the arbitrary $(d+1)$-subspaces of $V$ containing $\ell$, the $d$-subspaces of $H$ with the same property.  
    \end{itemize}
    To conclude, observe that projecting arbitrary lines in $H$ on $\pi\le H$ with $\dim \pi=d$ produces the same line as projecting on the $(d+1)$-space $\pi\oplus H^{\perp}$. 

  \item\label{item:th:sat.all:pf.dnd1n}\textbf{: $\cP_{d,n}\Rightarrow \cP_{d+1,n}$.} Effectively the same argument as before achieves this too: the ambient space $V$ can stay unaffected, while still projecting on $(d+1)$-spaces of the form $\pi\oplus \kappa$ for a line $\kappa$ orthogonal to $\ell$, $\ell'$ and any arbitrary other line $\ell''$ (ensuring that the latter belongs to the saturated set $S$).

    Jointly, \Cref{item:th:sat.all:pf.dnd1n1,item:th:sat.all:pf.dnd1n} and induction have reduced the problem to the already-settled \Cref{pr:sat.1.23}.  \qedhere
  \end{enumerate}
\end{th:sat.1.all}

There is a yet-more-general version of \Cref{th:sat.1.all}, covering arbitrary Grassmannians. That statement, though, cannot be \emph{precisely} the same.

\begin{example}\label{ex:spine}
  Fix $1\le r<d<n\in \bZ_{\ge 3}$, an $n$-dimensional Hilbert space, and fix
  \begin{equation*}
    \pi\in \bG(\le r,V)
    :=
    \bigsqcup_{1\le r'\le r}\bG(r',V).
  \end{equation*}
  The subset
  \begin{equation}\label{eq:uparr}
    \pi^{\uparrow}
    =
    \pi_{\bG(r,V)}^{\uparrow}
    :=
    \left\{\pi'\in \bG(r,V)\ :\ \pi'\ge \pi\right\}
    \subseteq
    \bG(r,V)
  \end{equation}
  (the \emph{$r$-spine of $\pi$}) is then $\tau_d$-saturated. 
\end{example}

%\newpage

The phenomenon recorded in \Cref{ex:spine}, it turns out, accounts for all saturation under the appropriate dimension constraints. 

\begin{theorem}\label{th:2rd.sat.r.all}
  For $1\le r< d<n\in \bZ_{\ge 3}$ the following conditions are equivalent.
  \begin{enumerate}[(a),wide]
  \item\label{item:th:2rd.sat.r.all:2rd} $2r\le d$. 

  \item\label{item:th:2rd.sat.r.all:spines} For any $n$-dimensional real or complex Hilbert space $V$, a $\tau_d$-saturated subset $S\subseteq \bG(r,V)$ is of the form
  \begin{equation*}
    S=\left(\bigcap S\right)^{\uparrow}
    \quad
    \left(\text{in the notation of \Cref{eq:uparr}}\right).
  \end{equation*}
  Consequently,
  \begin{equation*}
    \bG(\le r,V)
    \ni
    \pi
    \xmapsto{\quad}
    \pi^{\uparrow}
    \in
    2^{\bG(r,V)}
  \end{equation*}
  is a bijection between $\bG(\le r,V)$ and the collection of $\tau_d$-saturated sets of $r$-planes.  

  \item\label{item:th:2rd.sat.r.all:2el} For a $n$-dimensional real or complex Hilbert space $V$ the Grassmannian $\bG(r,V)$ contains no 2-element $\tau_d$-saturated subsets. 
  \end{enumerate}  
\end{theorem}

It will be convenient to modularize the proof, recording the formally non-obvious implications separately. We first dispose of what will later be incorporated as the \Cref{item:th:2rd.sat.r.all:2el} $\Rightarrow$ \Cref{item:th:2rd.sat.r.all:2rd} implication of \Cref{th:2rd.sat.r.all}.

\begin{lemma}\label{le:need.2rd}
  Let $1\le r< d<n\in \bZ_{\ge 3}$ and $V$ be a real or complex $n$-dimensional Hilbert space. If $2r>d$ the Grassmannian $\bG(r,V)$ contains 2-element $\tau_d$-saturated subsets. 
\end{lemma}
\begin{proof}
  Observe first that as soon as $2r>d$ (and hence $r>d-r$) we have
  \begin{equation*}
    n\ge 1+d = 1+(d-r)+r \ge 2+2(d-r),
  \end{equation*}
  so there are $\zeta,\zeta'\in \bG(r,V)$ with  
  \begin{equation*}
    \eta=\zeta\oplus \beta
    ,\quad
    \eta'=\zeta'\oplus \beta
    ,\quad
    \dim \zeta
    =
    \dim \zeta'>d-r
  \end{equation*}
  for mutually-orthogonal $\zeta$, $\zeta'$ and $\beta$. This means that $\zeta'$ (and hence $\eta'$) will intersect every $(n-d)$-subspace $\pi^{\perp}\le \eta^{\perp}$ nontrivially, so $P_{\pi}$ cannot inject on $\eta'$ . This also being valid with the roles of $\eta$ and $\eta'$ interchanged, the 2-element set $\{\eta,\eta'\}$ is $\tau_d$-saturated.
\end{proof}

The brunt of the work in proving \Cref{th:2rd.sat.r.all} consists in addressing the \Cref{item:th:2rd.sat.r.all:2rd} $\Rightarrow$ \Cref{item:th:2rd.sat.r.all:spines} implication, isolated as the following result. 

\begin{proposition}\label{pr:sat.r.all}
  Let $1\le r<2r\le d<n\in \bZ_{\ge 3}$ and $V$ an $n$-dimensional real or complex Hilbert space.

  A $\tau_d$-saturated subset $S\subseteq \bG(r,V)$ is of the form
  \begin{equation*}
    S=\left(\bigcap S\right)^{\uparrow}
    \quad
    \left(\text{in the notation of \Cref{eq:uparr}}\right),
  \end{equation*}
  and hence
  \begin{equation*}
    \bG(\le r,V)
    \ni
    \pi
    \xmapsto{\quad}
    \pi^{\uparrow}
    \in
    2^{\bG(r,V)}
  \end{equation*}
  is a bijection between $\bG(\le r,V)$ and the collection of $\tau_d$-saturated sets of $r$-planes. 
\end{proposition}
\begin{proof}
  That \emph{every} subset $S\subseteq \bG$ is contained in its attached $\left(\bigcap S\right)^{\uparrow}$ is tautological, so the crux here is the opposite inclusion $S\supseteq \left(\bigcap S\right)^{\uparrow}$ (assuming saturation). 
  
  Much of this will again be induction on size, so as in the proof of \Cref{th:sat.1.all}, it is convenient to have the shorthand $\cP_{r,d,n}$ for the desired conclusion. That proof's $\cP_{d,n}$ are now $\cP_{1,d,n}$, all available as inductive base cases, justifying the assumption that $r>1$ and all $\cP_{r',d',n'}$ are valid for $r'<r$.

  \begin{enumerate}[(I),wide]
  \item\label{item:pr:sat.r.all:pf.s.core}\textbf{We have
      \begin{equation*}
        \forall\left(S'\subseteq S\right)
        \left(
          \bigcap S'\ne \{0\}
          \xRightarrow{\quad}
          S\supseteq \left(\bigcap S'\right)^{\uparrow}
        \right).
      \end{equation*}
    } Indeed: denoting orthogonal complements by `$\ominus$' one may simply work only with spaces containing $\pi:=\bigcap S'$ or, equivalently, substitute $\bullet\ominus \pi$ for all spaces $\bullet$ involved. This reduces the problem to
    \begin{equation*}
      \cP_{r',d',n'}
      ,\quad
      n-n'=d-d'=r-r'=\dim \bigcap S'\ne 0
      \quad
      \left(\text{note: }2r'\le d'\right),
    \end{equation*}
    whence induction takes over. This already settles matters if $\bigcap S\ne \{0\}$, so we may henceforth assume $\bigcap S$ trivial and seek to prove that $S$ must be all of $\bG(r,V)$.

  \item\label{item:pr:sat.r.all:pf.prscr.int}\textbf{For $\eta\ne \eta'\in S$ the intersection $\left(P_{\pi}\eta'\right)\cap \eta$ can be any subspace of $\eta$ subject only to the obvious necessary constraint that it contain $\eta\cap \eta'$.} In other words, there is maximal freedom in prescribing the intersection $\left(P_{\pi}\eta'\right)\cap \eta$.

    Specifying a $d$-plane $\pi\ge \eta$ amounts to specifying
    \begin{equation*}
      \pi^{\perp}\in \bG(n-d,\eta^{\perp}). 
    \end{equation*}
    The latter space having codimension $d-r\ge r$ in the $(n-r)$-dimensional $\eta^{\perp}$, dominating the dimension of $\eta'$, we have $\eta'\cap \pi^{\perp}=\{0\}$ generically (over a dense \emph{Zariski}-open \cite[\S 2.1]{ms_nonl} set). By the same token, for $0\le r'\le r$ and a fixed $(r-r')$-dimensional $\zeta'\le \eta^{\perp}\ominus P_{\eta^{\perp}}\eta'$ we have 
    \begin{equation*}
      \bG\left(r',P_{\eta^{\perp}}\eta'\right)
      \ni
      \zeta
      \xmapsto{\quad}
      \left(P_{\left(\zeta\oplus \zeta'\right)^{\perp}}\eta'\right)\cap \eta
      \in
      \left(\eta\cap \eta'\right)_{\bG\left(r'+\dim \eta\cap\eta',\ \eta\right)}^{\uparrow}
    \end{equation*}
    identifies the collection of $r'$-spaces in $P_{\eta^{\perp}}\eta'$ with that of $(r'+\dim\eta\cap \eta')$-subspaces of $\eta$ containing the intersection $\eta\cap \eta'$ for $0\le r'\le r$. This is the stated target of \Cref{item:pr:sat.r.all:pf.prscr.int}, repackaged.

  \item\label{item:pr:sat.r.all:pf.is.triv.int}\textbf{If $\eta,\eta'\in S$ intersect trivially then for all lines $\ell\le \zeta$ we have $\ell^{\uparrow}_{\bG(r,V)}\subseteq S$.} Or: all $r$-planes containing $\ell$ are $S$-members. For arbitrary lines $\ell\le \eta$ we can find $\pi$ with $\left(P_{\pi}\eta'\right)\cap \eta=\ell$ by \Cref{item:pr:sat.r.all:pf.prscr.int}, so that indeed
    \begin{equation*}
      \forall\left(\ell\in \bP \eta\right)
      \forall\left(\eta''\in \ell^{\uparrow}_{\bG(r,V)}\right)
      \left(\eta''\in S\right)
    \end{equation*}
    by \Cref{item:pr:sat.r.all:pf.s.core} above.
    
  \item\label{item:pr:sat.r.all:pf.all.int.triv}\textbf{Every $\eta\in S$ intersects some $\eta'\in S$ trivially.} Suppose the smallest dimension of an intersection $\eta'\cap \eta$, $\eta'\in S$ is $0< r'< r$. We are assuming $\bigcap S\ne \{0\}$, so some $\eta''\in S$ will intersect $\eta$ along a proper non-zero subspace not containing $\eta\cap \eta'$. Now,
    \begin{equation*}
      \left(\eta\cap \eta'\right)_{\bG(r,V)}^{\uparrow}
      ,\quad
      \left(\eta\cap \eta''\right)_{\bG(r,V)}^{\uparrow}
      \quad
      \overset{\text{\Cref{item:pr:sat.r.all:pf.s.core}}}
      {\subseteq}
      \quad
      S,
    \end{equation*}
    and some two $r$-spaces $\zeta'$ and $\zeta''$ respectively in the two $(\bullet)^{\uparrow}$ sets will intersect along a subspace of dimension $<r'$. Further, some element of either $\left(\zeta'\cap \zeta''\right)_{\bG(r,V)}^{\uparrow}\subseteq S$ (if $\zeta'\cap \zeta'\ne \{0\}$, via \Cref{item:pr:sat.r.all:pf.s.core}) or $\ell^{\uparrow}_{\bG(r,V)}\subseteq S$ (if $\zeta'\cap \zeta''=\emptyset$, via \Cref{item:pr:sat.r.all:pf.is.triv.int} instead) will cut $\eta$ along a $(<r')$-dimensional space, providing a contradiction. 
    
  \item\textbf{Conclusion.} Per \Cref{item:pr:sat.r.all:pf.is.triv.int,item:pr:sat.r.all:pf.all.int.triv},
    \begin{equation*}
      \forall\left(\eta\in S\right)
      \forall\left(\ell\in \bP \eta\right)
      \left(\ell^{\uparrow}_{\bG(r,V)}\subseteq S\right).
    \end{equation*}
    That this entails $S=\bG(r,V)$ follows from the connectedness of the \emph{Grassmann graph} $\bG(r,V)$ \cite[Proposition 2.11]{pank_wign}, with $r$-planes as vertices, two joined by an edge precisely when their intersection is a hyperplane in both.
  \end{enumerate}
\end{proof}

The individual components are now ripe for aggregation.

\pf{th:2rd.sat.r.all}
\begin{th:2rd.sat.r.all}
  Simply observe that
  \begin{equation*}
    \text{\Cref{item:th:2rd.sat.r.all:2rd}}
    \xRightarrow{\quad\text{\Cref{pr:sat.r.all}}\quad}
    \text{\Cref{item:th:2rd.sat.r.all:spines}}
    \xRightarrow{\quad\text{formal}\quad}
    \text{\Cref{item:th:2rd.sat.r.all:2el}}
    \xRightarrow{\quad\text{\Cref{le:need.2rd}}\quad}
    \text{\Cref{item:th:2rd.sat.r.all:2rd}}
  \end{equation*}
  finishing the proof. 
\end{th:2rd.sat.r.all}

%%%%%%%%%%%%%%%%%%%%%%%%%%%%%%%% 
%%%%%%%%%%%%%%%%%%%%%%%%%%%%%%%% 

%\newpage

\addcontentsline{toc}{section}{References}
%\bibliography{bib}{}
%\bibliographystyle{plain}

% BEGIN INSERTED BBL (ternary-rel-grsmn-xv1.bbl)
\def\polhk#1{\setbox0=\hbox{#1}{\ooalign{\hidewidth
  \lower1.5ex\hbox{`}\hidewidth\crcr\unhbox0}}}

% END INSERTED BBL

\Addresses

\end{document}